\documentclass[10pt,a4paper,reqno]{amsart}
\usepackage[english]{babel}
\usepackage{amssymb,amsthm,amsmath} 
\usepackage{mathtools}
\usepackage{mathrsfs}
\usepackage{cite}
\usepackage{enumerate}
\usepackage{stackrel}
\usepackage{eucal}
\usepackage{color}
\usepackage{xcolor}
\usepackage{graphics}
\usepackage{subcaption}
\usepackage{hyperref}
\usepackage{soul}
\numberwithin{equation}{section}
\newcommand{\ie}{\emph{i.e.}}
\newcommand{\eg}{\emph{e.g.}}
\newcommand{\cf}{\emph{cf.}}

\newcommand{\Com}{\mathbb{C}}
\newcommand{\Real}{\mathbb{R}}

\newcommand{\sgn}{\mathop{\mathrm{sgn}}\nolimits}

\newcommand{\Tr}{\mathop{\mathrm{Tr}}\nolimits}

\newcommand{\Dom}{\mathsf{D}}
\newtheorem{Theorem}{Theorem}
\newtheorem{Corollary}{Corollary}
\newtheorem{Lemma}{Lemma}
\theoremstyle{definition}

\newtheorem{Remark}{Remark}

\newcommand{\eps}{\varepsilon}

\newcommand{\la}{\lambda}
\newcommand{\RR}{{\mathbb{R}}}
\newcommand{\CC}{{\mathbb{C}}}

\newcommand\mydot{\,\cdot\,}
\newcommand{\ov}{\overline}
\newcommand{\ii}{{\rm i}}

\usepackage{color}

\renewcommand{\d}{\mathrm{d}}
\newcommand{\Hm}[1]{\leavevmode{\marginpar{\tiny%
			$\hbox to 0mm{\hspace*{-0.5mm}$\leftarrow$\hss}%
			\vcenter{\vrule depth 0.1mm height 0.1mm width \the\marginparwidth}%
			\hbox to
			0mm{\hss$\rightarrow$\hspace*{-0.5mm}}$\\\relax\raggedright #1}}}

\newcommand\I{\mathrm{i}}
\newcommand{\R}{{\mathbb R}}
\newcommand{\C}{{\mathbb C}}
\newcommand{\N}{{\mathbb N}}

\newcommand\im{\mathrm{Im}\,}	
\newcommand{\e}{{\rm e}}	
\newcommand{\rd}{{\rm d}}
\title[]{Sharp spectral bounds for complex perturbations of the 
	indefinite Laplacian}

\begin{document}



\author{Jean-Claude~Cuenin}
\address{Department of Mathematical Sciences, 
	Loughborough University, Loughborough,
	Leicestershire, LE11 3TU United Kingdom}
\email{J.Cuenin@lboro.ac.uk}
\author{Orif~O.~Ibrogimov}
\address{Institute for Theoretical Physics, ETH Z\"urich,
	Wolfgang-Pauli-Str.~27, 
	8093 Z\"urich, 
	Switzerland}
\email{oibrogimov@phys.ethz.ch}




\date{April 22, 2020}


\begin{abstract}
We derive quantitative bounds for eigenvalues of complex 
perturbations of the indefinite Laplacian on the real line. 
Our results substantially improve existing results even for 
real potentials. 
For $L^1$-potentials, we obtain optimal spectral enclosures 
which accommodate also embedded eigenvalues, while our result 
for $L^p$-potentials yield sharp spectral bounds on the  
imaginary parts of eigenvalues of the perturbed operator
for all $p\in[1,\infty)$. 
The sharpness of the results are demonstrated by means of explicit examples.
\end{abstract}

\maketitle
\section{Introduction and main results}
\noindent
Spectral estimates for non-self-adjoint Schr\"odinger operators 
and related questions have been a very active area of mathematical 
research recently. By now several robust methods have been developed
and the corresponding literature is extensive. To name a few, we mention~\cite{Boegli-CMP-17,Cue-Ken-CPDE17,Dav-Nath02,Enblom-LMP-16,Fan-Kre-Veg-JST18,Frank-BLMS-11,Frank-TAMS-18,Lap-Saf-CMP09,Lee-Seo-JMAA-19,Fra-Lap-Lieb-Sei-LMP06, Safr-BLMS10,Fan-Kre-Veg-JFA-2018, Hen-Kre-JST2017, Cuenin-JFA-17, Krej-Siegl-JFA2019,Ibr-Sta-19,Cuenin-improved-2019,Krej-Siegl-JFA2019,Coss2019,Cue-Lap-Tre-14, Fan-Kre-LMP19, Cuenin-Siegl-LMP-18, Cas-Ibr-Kre-Sta-inprep} and also \cite{Cue-Tre-JMAA16,Dem-Hans-Kat-JFA-09,Dem-Hans-Kat13} for related recent results obtained in an abstract operator theoretic setting.

To goal of this paper is to establish sharp spectral estimates for the operator
\begin{equation}\label{indefSL-intro}
H_V:= \sgn(x)(-\partial^2_x\,+\,V) 
\qquad \mbox{in} \qquad L^2(\RR)\,.
\end{equation}

%
The unperturbed operator~$H_0$ is not
symmetric in the Hilbert space 
$L^2(\RR)$ 
due to the sign change of the weight function; it can be 
interpreted as a self-adjoint operator with respect to the 
Krein space inner product
$(\operatorname{sgn}\cdot,\cdot)$ in $L^2(\RR)$. The sum in \eqref{indefSL-intro} is defined in the form sense (see Section~\ref{Sec:op.def} for details)

Unlike its definite counterpart (i.e.\ without the sign function), the existing literature on  
eigenvalue bounds for operators of type~\eqref{indefSL-intro} 
seems to be much more sparse even in the one-dimensional
setting with real potentials, see \eg~\cite{Beh-Schm-Tru-2019, Beh-Schm-Tru-2017, Beh-Kat-Tru-2009},
and a number of interesting questions remain open.
One of them is a conjecture of Behrndt in \cite{Beh-open-2013}
according to which the eigenvalues of singular indefinite 
Sturm-Liouville operators (\ie~\eqref{indefSL-intro} with 
real-valued~$V$) accumulate to the real axis whenever
the eigenvalues of the corresponding definite Sturm-Liouville 
operator accumulate to the bottom of the essential spectrum from 
below. So far the conjecture of Behrndt has been confirmed only for a particular family of (shifted Coulomb type) potentials in \cite{Lev-Seri-2016}.

The present paper is partially motivated by the recent 
work \cite{Beh-Schm-Tru-2018}, where the following (non optimal) indefinite analogue of the celebrated result of \cite{Abr-Asl-Dav-01} 
was shown: every non-real eigenvalue $\la$ of the singular indefinite 
Sturm-Liouville operator with a real-valued 
potential $V\in L^1(\RR)$ obeys the estimate
\begin{equation}\label{BST.bound}
|\la|^{1/2}\leq 
\|V\|_1\,.
\end{equation}

Our  standing assumption is that 
there exist numbers $a\in(0,1)$ and $b\in\Real$ such that, 
for all $\psi\in H^1(\Real)$,
\begin{equation}\label{rel.bddness}
\int_{\Real} |V||\psi|^2 \leq a\int_{\Real} |\psi'|^2 
+ b\int_{\Real} |\psi|^2\,.
\end{equation}
In particular, potentials $V \in L^p(\Real)$ are covered 
by this hypothesis for all $p\in[1,\infty]$, 
see Section~\ref{Sec:op.def} for more details.
Furthermore, in the same way, 
it is also possible to proceed in a greater generality
and give a meaning to the distributional Dirac delta potential~$\delta$,
which is explicitly solvable.
Putting $\|\delta\|_{L^1(\Real)}:=1$ by convention,
Theorem~\ref{thm:ell1} remains valid in this 
more general setting.
The feature of our results are that we work with complex 
potentials of minimal regularity, our bounds are quantitative 
and substantially improve existing results. Moreover, 
for $L^1$-potentials our spectral enclosures are 
sharp and we cover embedded eigenvalues, too.

Throughout the paper we denote by $\|\cdot\|_p$ the standard 
	$L^p$-norm for $p\in[1,\infty]$.
Our first result quantitatively improves the known 
bound~\eqref{BST.bound} of \cite{Beh-Schm-Tru-2018} and 
holds also for possibly embedded eigenvalues.

\begin{Theorem}\label{thm:ell1}
	For $V\in L^1(\Real)$, every eigenvalue~$\la$ of $H_V$ satisfies
	\begin{equation}\label{encl.ell1}
	\sqrt{2}|\la|\leq \sqrt{|\la|+|\Re(\la)|}\,\|V\|_1\,.
	\end{equation}
\end{Theorem}
%
\begin{Remark}
	\begin{enumerate}[\upshape(i)]
		\item The bound in \eqref{encl.ell1} is sharp in the sense 
		that, for any $Q>0$ and any point~$\la\in\CC\setminus\RR$ which 
		fulfills the equation
		\[
		2|\la|^2=\bigl(|\la|+|\Re(\la)|\bigr)\,Q^{2}\,,
		\]
		there exists $V\in L^1(\Real)$ such that $Q=\|V\|_1$ 
		and~$\lambda$ is an eigenvalue of the corresponding 
		operator~$H_{V}$, see Section~\ref{Sect.optim}. This means that 
		every boundary point of the spectral enclosure corresponding 
		to~\eqref{encl.ell1}, with the exception of points located 
		in the real line, is an eigenvalue of~$H_V$ for some 
		$L^1$--potential~$V$. Hence the obtained spectral enclosure 
		cannot be squeezed any further. 
		\item It readily follows from Theorem~\ref{thm:ell1} that the intervals
		$(-\infty, \|V\|_1^2)$ and $(\|V\|_1^2, \infty)$ are free of 
		embedded eigenvalues of $H_V$ for $V\in L^1(\Real)$.
		\item In view of the elementary inequality $|\Re(\la)|\leq|\la|$, 
		it is obvious that the bound~\eqref{encl.ell1} provides a strictly tighter 
		spectral enclosure than \eqref{BST.bound}. 
	\end{enumerate}
\end{Remark}
\begin{Corollary}\label{Cor:ineq.imag.L1}
	For $V\in L^1(\Real)$, every eigenvalue~$\la$ of $H_V$ satisfies
	the sharp bound
	\begin{equation}\label{sharp.ineq.L1.imag}
	|\Im(\la)|\leq\frac{3\sqrt{3}}{8} \,\|V\|_1^2\,.	
	\end{equation}
\end{Corollary}

The spectral enclosure corresponding to \eqref{encl.ell1} 
is a compact set, which is symmetric with respect to both the real 
and the imaginary axes. The geometry of its boundary is 
quite easy to understand. We refer to Figure~\ref{fig:1}
for the plots of the boundary curves corresponding 
to~\eqref{BST.bound} and \eqref{encl.ell1} with $\|V\|_1=1$.

\begin{figure}[htb!]
	\centering
	\begin{subfigure}[c]{0.49\textwidth}
		\includegraphics[width=\textwidth, height=\textwidth]{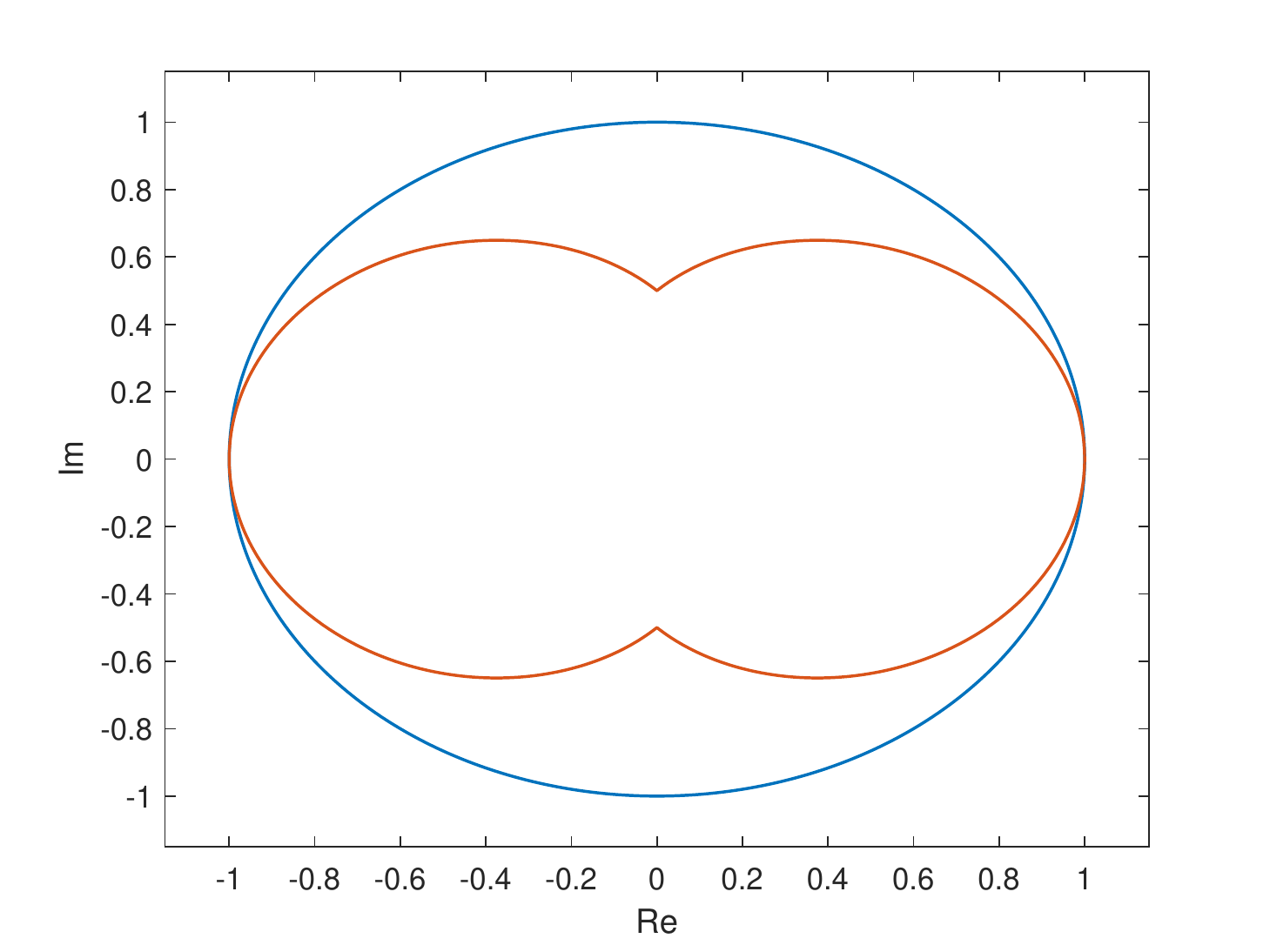}
	\end{subfigure}
	\begin{subfigure}[c]{0.49\textwidth}
		\includegraphics[width=\textwidth, height=\textwidth]{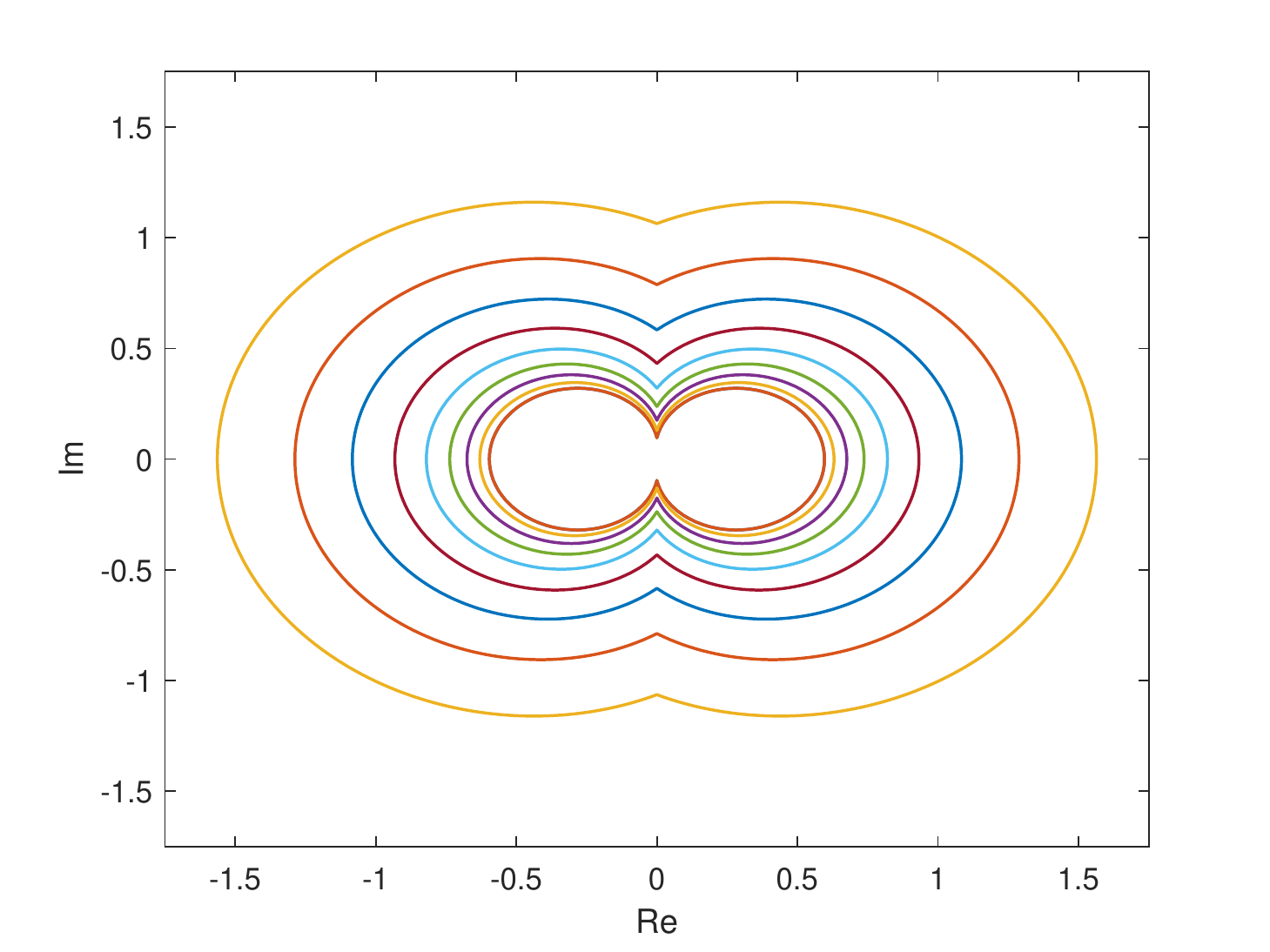}
	\end{subfigure}
	\caption{\small{The blue circle and the red curve on the left 
			correspond to the boundary curves of the spectral enclosures 
			respectively given by~\eqref{BST.bound} and~\eqref{encl.ell1} 
			for~$\|V\|_{L^1(\RR)}=1$.
			The picture on the right illustrates plots of the expanding boundary 
			curves corresponding to the spectral 
			enclosure~\eqref{encl.ell1} for 
			$\|V\|_1=(1.35)^j/7$, $j=1,2,\ldots,9$.}}
	\label{fig:1}
\end{figure}

Our next result provides a spectral estimate in terms of the 
$L^p$-norm of the potential for $p\in(1,\infty)$. The proof interpolates the bound of Theorem \ref{thm:ell1} with the standard operator norm bound for the resolvent of a self-adjoint operator. The strategy was first used in \cite{Frank-TAMS-18}.

\begin{Theorem}\label{thm:p-norm}
	Let $p\in(1,\infty)$. For $V\in L^p(\RR)$, every eigenvalue~$\la$ 
	of $H_V$ satisfies
	\begin{equation}\label{encl.ellp}
	2^{\frac{3}{2p}-1}|\la|^{\frac{1}{p}}|\Im(\la)|^{1-\frac{1}{p}} \leq 
	\bigl(|\la|+|\Re(\la)|\bigr)^{\frac{1}{2p}}\|V\|_p\,.
	\end{equation}
\end{Theorem}
\begin{Remark}
	\begin{enumerate}[\upshape (i)]
		\item The bound \eqref{encl.ellp} trivially holds 
		for embedded eigenvalues. 
		Letting $p\searrow1$~in~\eqref{encl.ellp}, one arrives 
		at the bound \eqref{encl.ell1}. An analogous version of
		the bound \eqref{encl.ellp} for the case $p=\infty$ is 
		very easy to derive and reads as 
		\[
		|\Im(\la)|\leq 2\|V\|_{\infty}\,.
		\] 
		\item Using the elementary inequalities 
		$|\Re(\la)|\leq|\la|$ and $|\Im(\la)|\leq|\la|$, 
		one can easily read off from \eqref{encl.ellp} that
		the following rough bound holds
		%
		%
		\begin{equation}\label{bound.imag.Lp}
		|\Im(\la)|\leq 2^{\frac{2p-2}{2p-1}}\|V\|^{\frac{2p}{2p-1}}_p\,.
		\end{equation}
		This improves the recent result in \cite{Phil-2019}
		where, for real-valued $V\in L^p(\RR)$ with $p\geq 2$, the estimate
		\eqref{bound.imag.Lp} was shown to hold with a constant 
		factor $C(p)$ such that $C(p)>2^{\frac{2p-2}{2p-1}}$.
		While the result of 
		\cite{Phil-2019} is a considerable improvement of 
		the analogous result of \cite{Beh-Schm-Tru-2019}, 
		the function $C(p)$ therein is a strictly decreasing 
		function of the parameter~$p$ with the limit equal to 
		$2+\sqrt{2}$ at infinity.
	\end{enumerate}
\end{Remark}
As a matter of fact, by an optimization trick under the 
constraint \eqref{encl.ellp}, one can substantially improve 
the bound \eqref{bound.imag.Lp}, thus generalizing the 
sharp bound \eqref{sharp.ineq.L1.imag} for $L^p$-potentials 
for all $p\in(1,\infty)$.
\begin{Corollary}\label{Cor:ineq.imag}
	Let $p\in(1,\infty]$. For $V\in L^p(\RR)$, every eigenvalue~$\la$ 
	of $H_V$ satisfies the bound
	\begin{equation}\label{ineq.imag}
	|\Im(\la)|\leq 2\Bigl(\frac{3\sqrt{3}}{16}\Bigr)^{\frac{1}{2p-1}}
	\|V\|^{\frac{2p}{2p-1}}_p\,.
	\end{equation}	
\end{Corollary}

Below we provide various plots of the expanding boundary 
curves corresponding to the spectral enclosure from 
Theorem~\ref{thm:p-norm}. 
\begin{figure}[htb!]
	\centering
	\includegraphics[width=0.65\textwidth]{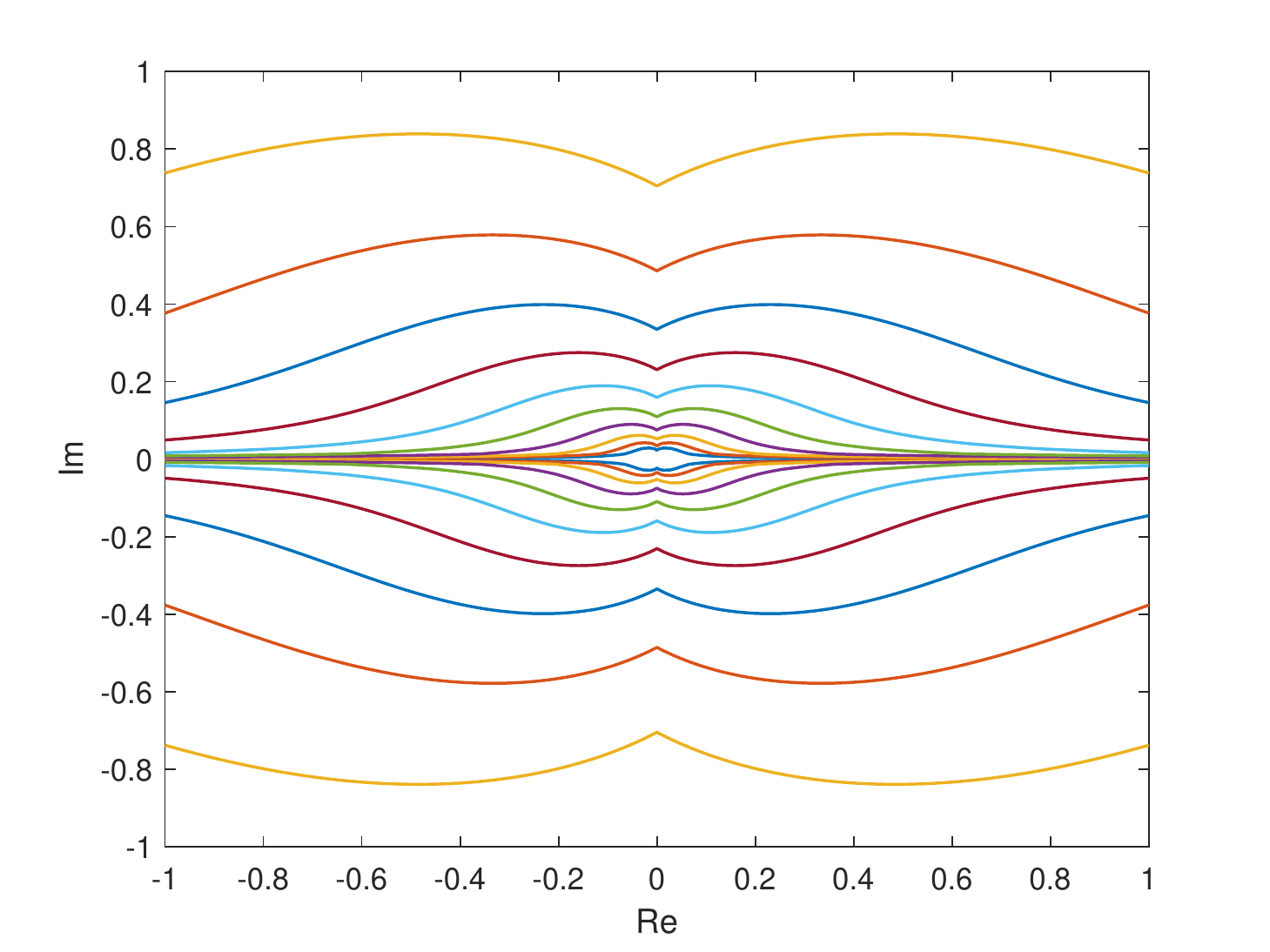}
	\caption{The plots of the expanding boundary curves 
		corresponding to the spectral enclosure from 
		Theorem~\ref{thm:p-norm} for $p=1.25$ and 
		$\|V\|_p=(1.25)^j/10$, $j=1,2,\dots,10$.}
	\label{fig:sp_thm2}
\end{figure}

The next result concerns with spectral bounds for 
$L^p$-potentials again. 
This time, however, it is given in terms of 
the $L^p$-sizes of the restrictions to the
intervals $(0,\infty)$ and $(-\infty, 0)$ 
of the potential.
In the sequel, for $p\in(1,\infty)$, we denote 
by $\|V\|_{p,\pm}$ the standard $L^p$-norms over 
the intervals $(0,\infty)$ and $(-\infty, 0)$, respectively.
\begin{Theorem}\label{thm:p-norm.1}
	Let $p\in(1, \infty)$. For $V\in L^p(\RR)$, every non-real eigenvalue of $H_V$ satisfies
	\begin{equation}\label{spec.ineq.main}
	\sqrt{2}\,\Bigl(\frac{p}{p-1}\Bigr)^{1-\frac{1}{p}}\sqrt{|\la|} \leq 
	\max\Bigg\{\frac{\|V\|_{p,-}}{\bigl|\Re(\sqrt{\la})\bigr|^{1-\frac{1}{p}}}
	+\frac{\sqrt{2}\|V\|_{p,+}}{\bigl|\Im(\sqrt{\la})\bigr|^{1-\frac{1}{p}}}\,,
	\;\, \frac{\sqrt{2}\|V\|_{p,-}}{\bigl|\Re(\sqrt{\la})\bigr|^{1-\frac{1}{p}}}
	+\frac{\|V\|_{p,+}}{\bigl|\Im(\sqrt{\la})\bigr|^{1-\frac{1}{p}}}\Bigg\}\,.
	\end{equation} 
\end{Theorem}
\begin{Remark}
	\begin{enumerate}[\upshape(i)]
		\item In general, neither of the spectral bounds of Theorems~\ref{thm:p-norm} and \ref{thm:p-norm.1} is better than the other. Numerical experiments, however, indicate that the one of Theorem~\ref{thm:p-norm} is better than that of Theorem~\ref{thm:p-norm.1} for sufficiently small values of $p>1$, while the opposite is true for sufficiently large values of $p$.
		\item Unlike that of Theorem~\ref{thm:ell1}, the spectral enclosures of Theorems~\ref{thm:p-norm} and \ref{thm:p-norm.1} are non-compact sets, yet being symmetric with respect to both the real 
		and the imaginary axes. 		
	\end{enumerate}
\end{Remark}

Note that the bounds \eqref{BST.bound}, \eqref{encl.ell1}--\eqref{encl.ellp} are scale-invariant. More precisely, if $U_{\rho}$ is the unitary operator $U_{\rho}f(x)=\rho^{-1/2}f(x/\rho)$, $\rho>0$, then
\begin{align}\label{scaling}
U_{\rho}^*H_VU_{\rho}=\rho^{-2}\sgn(x)(-\partial_x^2+V_{\rho}),\quad V_{\rho}(x)=\rho^{2}V(\rho x)\,,
\end{align}
has the same eigenvalues as $H_V$. Hence,
\begin{align*}
\lambda\in\sigma_{\rm p}(H_V)\iff \rho^{2}\lambda\in\sigma_{\rm p}(H_{V_{\rho}})\,.
\end{align*}
Since $\|V_{\rho}\|_p=\rho^{2-1/p}\|V\|_p$ and $|\cdot|$, $\Im(\cdot)$, 
$\Re(\cdot)$ are homogeneous of degree one, applying any of the 
inequalities \eqref{BST.bound}, \eqref{encl.ell1}--\eqref{encl.ellp} to $H_{V_{\rho}}$ instead 
of $H_V$ does not change the outcome. 

Next, by a Wigner-von Neumann type example we show that the scale-invariant inequality
\begin{align}\label{LT p>1}
|\lambda|^{p-\frac{1}{2}}\leq C\|V\|_p^p\,,
\end{align}
which holds for $p=1$ by \eqref{BST.bound}, cannot hold for $p>1$. The idea to consider complex-valued  Wigner-von Neumann potentials in the context of spectral estimates for non-self-adjoint operators is due to Frank and Simon \cite{Fra-Sim-JST-17}. 

%
%
\begin{Theorem}\label{thm:Wigner-vonNeumann}
	Let $\lambda>0$ be arbitrary. There exists a positive constant 
	$C=C(\la)$ and a sequence of potentials $V_n:\R\to\R$ such that,
	for all $n$, $\la$ is an eigenvalue of $H_{V_n}$ and
	\begin{align}\label{bound Vn}
	|V_n(x)|\leq \frac{C}{n+|x|}\,, \quad x\in\RR\,.
	\end{align}
	In particular, for any $p>1$, we have
	\begin{align*}
	\lim_{n\to\infty}\|V_n\|_p=0\,.
	\end{align*}
\end{Theorem}

We now modify the above result to produce a \emph{non-real} eigenvalue 
instead of an embedded one, showing that the scale-invariant 
inequality~\eqref{LT p>1} also fails in this case. 
Additionally, we demonstrate that the exponents 
in~\eqref{encl.ellp} cannot be improved.

\begin{Theorem}\label{thm:square.well}
	Given $\eps>0$ sufficiently small and $\mu>0$, there exists a 
	potential $V=V(\eps,\mu)\in L^{\infty}_c(\R)$ such that~$H_V$ has 
	eigenvalue $\la=\la(\eps,\mu)$ with 
	$\Re(\la)=\mu(1+\mathcal{O}(\eps))$, 
	$\Im(\la)=\mu\eps(2+\mathcal{O}(\eps^2))$, 
	and for every $p\geq 1$,
	\begin{align}\label{square well limit V}
	\|V(\eps,\mu)\|_p\approx 
	|\mu|^{1-\frac{1}{2p}}\eps^{1-\frac{1}{p}}|\ln\eps|^{\frac{1}{p}}\,.
	\end{align}
	In particular, for $p>1$, we have
	\begin{align*}
	\lim_{\eps\to 0}\|V(\eps,\mu)\|_p=0\,.
	\end{align*}
\end{Theorem}
\begin{Remark}
	Theorem~\ref{thm:square.well} can be reformulated in 
	the following way: Consider the sector
	\begin{align*}
	\Sigma_{\eps_0}:=\{\la\in\C:\, |\Im(\la)|\leq \eps_0|\Re(\la)|\}\,,
	\end{align*}
	where $\eps_0$ is sufficiently small but fixed. Then, given $\lambda$ 
	in the intersection of $\Sigma_{\eps_0}$ with the first quadrant, 
	there exists a potential $V=V(\la)\in L^{\infty}_c(\R)$ such 
	that $H_V$ has eigenvalue~$\la$ and, for every $p\geq 1$,
	\begin{align*}
	\|V(\la)\|_p\approx |\la|^{\frac{1}{2p}}|\Im(\la)|^{1-\frac{1}{p}}
	\bigg|\ln\bigg(\frac{|\Im(\la)|}{|\la|}\bigg)\bigg|^{\frac{1}{p}}.
	\end{align*}
	This is easily seen by observing that we can find $\la=\la(\eps,\mu)$ 
	as in the theorem with
	\begin{align*}
	\mu=\Re(\la)(1+\mathcal{O}(\eps_0)), \quad 
	\eps=\frac{\Im(\la)}{2\Re(\la)}(1+\mathcal{O}(\eps_0)), \quad 
	|\la|=\Re(\la)(1+\mathcal{O}(\eps_0))\,.
	\end{align*}
	It is also clear that the assumption of $\la$ lying in the first 
	quadrant can be omitted.
\end{Remark}

In the next result we study the weak coupling limit, i.e. we 
replace $V$ by $\eps V$ in $H_V$ and establish existence, uniqueness 
and asymptotics of an eigenvalue $\la(\eps)$ as $\eps\to 0^+$. 
This will yield another confirmation that the spectral bound 
\eqref{encl.ell1} is sharp.
In the following we set
\begin{align*}
v_+=\int_{0}^{\infty}V(x)\,\rd x, \quad 
v_-=\int_{-\infty}^{0}V(x)\,\rd x, \quad 
v_{\rm sgn}=\int_{-\infty}^{\infty}\sgn(x)V(x)\,\rd x\,.
\end{align*}
\begin{Theorem}\label{thm:weak.coupling}
	Assume that $\Re(v_{\rm sgn})+\Im(v_{\rm sgn})<0$ and 
	$\Re(v_{\rm sgn})<\Im(v_{\rm sgn})$. Then, for all sufficiently 
	small $\eps >0$, there exists a unique eigenvalue 
	$\la(\eps)\in\C^+$ of $H_{\eps V}$ satisfying
	\begin{equation*}
	\la(\eps) =\frac{\eps^2}{(1-i)^2}v_{\rm sgn}^2+ o(\eps^2), \quad \eps \to 0^+.
	\end{equation*}
\end{Theorem}

Our next result is a particular case of a Lieb-Thirring type bound. 
This is an analogue to a special case of \cite[Thm~1.3]{Frank-TAMS-18} 
for the definite Schr\"odinger operator. In fact, the conclusions 
of \cite[Thm~1.2]{Frank-TAMS-18} and \cite[Thm~1.3]{Frank-TAMS-18} 
in the case $d=1$ there continue to hold for the indefinite operator 
$H_V$ considered here. 
%
\begin{Theorem}\label{thm:LT}
	For $V\in L^1(\R)$, we have 
	\begin{align*}
	\sum_j |\Im(\lambda_j)|\leq C\|V\|_1^2\,,
	\end{align*}
	where $\lambda_j$ are the eigenvalues of $H_V$ repeated according to 
	their algebraic multiplicities.
\end{Theorem}

We use the method of \cite{Fra-Lap-Saf-2016} to bound the number 
of eigenvalues of $H_V$ in terms of an exponentially 
weighted~$L^1$ norm of~$V$ as follows.

\begin{Theorem}\label{thm:nr.of.EV}
	The number of eigenvalues $N(V)$ of $H_V$, counting algebraic 
	multiplicities, satisfies, for any $\eps>0$,
	\begin{align*}
	N(V)\leq \frac{1}{\eps^2}\left(\int_{\R}\e^{\eps|x|}|V(x)|\,\rd x\right)^2\,.
	\end{align*}
\end{Theorem}

The outline of the paper is as follows. In Section~\ref{Sec:op.def}, we rigorously introduce the perturbed operator~\eqref{indefSL-intro} using form methods. In Section~\ref{Sec:free.resolvent}, we derive a sharp estimate on the integral kernel of the free resolvent as well as a bound on a Krein type resolvent. In Section~\ref{Sec:BS.principle}, we establish a limiting absorption principle; this can be viewed as a one-sided conventional Birman-Schwinger principle which holds also for embedded eigenvalues. The proofs of Theorems~\ref{thm:ell1}-\ref{thm:nr.of.EV} and Corollaries~\ref{Cor:ineq.imag.L1}, \ref{Cor:ineq.imag} are given in Section~\ref{Sec:proofs}.
\section{Definition of the operator $H_V$}\label{Sec:op.def}
\noindent
Let~$T_0$ be the the self-adjoint operator in $L^2(\Real)$ 
associated with the quadratic form 
\begin{equation}
t_0[\psi]:=\int_{\Real} |\psi'|^2\,, 
\qquad 
\Dom(t_0):=H^1(\Real)\,.
\end{equation}
One has $\Dom(T_0) = H^2(\Real)$ and $T_0=-\partial^2_x$. 
The spectrum of~$T_0$ is purely absolutely continuous
and coincides with the semi-axis $[0,+\infty)$.

Let~$v$ be a quadratic form in $L^2(\Real)$, 
which is relatively bounded with respect to~$t_0$ 
with the relative bound less than one.
That is, $\Dom(v) \supset H^1(\Real)$
and there exist numbers $a\in(0,1)$ and $b\in\Real$ 
such that, for all $\psi\in H^1(\Real)$,
\begin{equation}\label{rb}
|v[\psi]| \leq a\int_{\Real}| \psi'|^2
+b\int_{\Real} |\psi|^2\,.
\end{equation}
Then the sum $t_V := t_0+v$ is a closed sectorial 
form with $\Dom(t_V)=H^1(\Real)$, which gives rise 
to an m-sectorial operator~$T_V$ in $L^2(\Real)$ 
via the representation theorem 
(\cf~\cite[Thm.~VI.2.1]{Kato}). 

For example, if $V \in L^1_\mathrm{loc}(\Real)$ is such that
\[
v[\psi]:=\int_{\Real} V |\psi|^2\,, 
\qquad
\Dom(v):=
\left\{
\psi\in L^2(\Real):\int_{\Real}|V||\psi|^2<\infty
\right\},
\]
verifies~\eqref{rb} (which coincides with~\eqref{rel.bddness} 
in this case), then we write $T_V = T_0 \dot{+} V$.
%
%
%
%

Let us now discuss sufficient conditions which guarantee~\eqref{rb}.

By the Sobolev embedding theorem (\cite[Thm.~5.4]{Adams}),
every function $\psi \in H^1(\Real)$ is bounded and continuous.
More specifically (\cf~\cite[Theorem~IX.28]{RS2}),
for any positive~$\alpha$ there is $\beta \in \Real$ 
such that, for all $\psi \in H^1(\Real)$,
\begin{equation}\label{elliptic}
\|\psi\|_{\infty}
\leq \alpha\|\psi'\|_2 + \beta\|\psi\|_2\,.
\end{equation}
Consequently, any potential $V \in L^1(\Real) + L^\infty(\Real)$
satisfies~\eqref{rel.bddness} with the relative bound equal to zero
(\ie~$a$ can be chosen arbitrarily small).  
\section{Free resolvent and a Krein type resolvent formula}\label{Sec:free.resolvent}
For real-valued $L^p$-potentials, the perturbed operator 
can be viewed as a self-adjoint operator in the 
Krein space with the indefinite inner product
$(\operatorname{sgn}\cdot,\cdot)$. Consequently, 
the spectrum of the perturbed operator is symmetric with 
respect to the real axis, see \eg~\cite{Kara-Trunk-2009, Curgus-Langer-1989}.
Certainly, this property no longer holds for general 
complex-valued potentials. As it was mentioned in the 
Introduction, it turns out, however, that our spectral
enclosures are symmetric with respect to the both
the real and the imaginary axes. Unless specified otherwise, 
for the rest of the paper we will work with a fixed spectral 
parameter~$\la$ from the \emph{upper half-plane} $\Com^+$ 
since the analysis of the lower part of the spectrum 
of $H_V$ is identical to that of the part in the upper half-plane.

\subsection{Free resolvent}

In the Hilbert space $L^2(\Real)$, let us consider the 
unperturbed operator 
\begin{equation*}
H_0:= \sgn(x)(-\partial^2_x)\,, \quad \Dom(H_0):= H^2(\Real)\,.
\end{equation*}
The spectrum of $H_0$ is continuous and coincides with~$\RR$.
For $\la\in\CC^+$, we denote by $G_\la$ the 
\emph{Green's function} of $H_0-\la$, \ie~the integral kernel 
of the free resolvent~$(H_0-\la)^{-1}$, which can be 
determined explicitly using the well-known form of the 
Green's function for the definite counterpart of $H_0$. 
In fact, we have\footnote{Here and in the sequel we choose 
	the branch of the square root with $\Im(\sqrt{\lambda})>0$, $\Re(\sqrt{\lambda})>0$ ($\lambda\in \CC^+$).} 
\begin{equation}\label{Green.fn}
G_\la(x,y)=
\frac{1}{2\alpha\sqrt{\la}}\begin{cases}
\alpha e^{i\sqrt{\la}(x+y)} 
+\ov{\alpha} e^{i\sqrt{\la}|x-y|} 
&\quad x\geq 0,\, y\geq 0,\\
-e^{\sqrt{\la}(ix+y)} 
&\quad x\geq 0,\, y<0,\\
e^{\sqrt{\la}(x+iy)} 
&\quad x<0,\, y\geq 0,\\
-\ov{\alpha}e^{\sqrt{\la}(x+y)}
-\alpha e^{-\sqrt{\la}|x-y|}
&\quad x<0,\,y<0,
\end{cases}
\end{equation}
where $\alpha:=\frac{1-i}{2}$ (see also \cite{Beh-Schm-Tru-2018}). 
Observe that, for all non-zero $x,y$, we have 
\begin{equation}\label{assym.Green.fn}
\sgn(x)G_\la(x,y)=\sgn(y)G_\la(y,x)\,.
\end{equation}

The pointwise estimate of the Green's function obtained 
in the next lemma plays a crucial role in establishing 
Theorem~\ref{thm:ell1}. 
\begin{Lemma}\label{Lem.Resv.estim}
	Let $\la\in\CC^+$. The Green's function 
	in \eqref{Green.fn} obeys the sharp pointwise 
	estimate
	\begin{equation}\label{Ineq.Green.fn}
	|G_\la(x,y)|^2 
	\leq
	\frac{1}{2|\la|}+\frac{|\Re(\la)|}{2|\la|^2}\,.
	\end{equation}
\end{Lemma}
\begin{proof}
	Let $a:=\Re(\sqrt{\la})>0$ and $b:=\Im(\sqrt{\la})>0$.
	We distinguish the four cases.
	
	\medskip
	\noindent
	\textbf{Case} $x\geq0$, $y\geq 0$.
	In view of \eqref{assym.Green.fn}, there is no loss of generality 
	in assuming that $x\geq y$. Then elementary calculations yield
	\begin{equation}\label{G.la.Phi}
	|G_\la(x,y)|^2 = \frac{1}{4|\la|} \Phi(x,y)\,,
	\end{equation}
	where 
	\begin{equation}\label{Phi}
	\Phi(x,y):= e^{-2b(x+y)}+e^{-2b(x-y)}+2e^{-2bx}\sin(2ay)\,. 
	\end{equation}
	We fix $y\geq0$ 
	and consider $\Phi$ as a 
	function of the variable $x$ on the interval $[y,\infty)$. 
	Since~$\Phi$ is non-negative for obvious reasons, and 
	\begin{equation}
	\frac{\partial}{\partial x}\Phi(x,y)=-2b\Phi(x,y) \leq0\,,
	\end{equation}
	we conclude that $\Phi(\mydot,y)$ is non-increasing 
	on~$[y,\infty)$. Therefore,
	\begin{equation}\label{def.phi}
	\Phi(x,y) \leq \Phi(y,y)
	=1+e^{-4by}+2e^{-2by}\sin(2ay)=:\varphi(y)\,.
	\end{equation}
	In view of the identities $\Re(\la)=a^2-b^2$ and 
	$\Im(\la)=2ab$, it thus suffices to show that 
	\begin{equation}\label{suff.cond.phi}
	\max_{y\geq0}\varphi(y)
	\leq 2\bigg(1+\frac{|a^2-b^2|}{a^2+b^2}\bigg)\,.
	\end{equation}
	It is not difficult to check that  
	\begin{equation}\label{cr.pt.eqn}
	\varphi'(y)=0 \quad \Longleftrightarrow 
	\quad e^{-2by}=\frac{a}{b}\cos(2ay)-\sin(2ay)\,.
	\end{equation}
	In view of this, elementary calculations show that, if $y_0\geq0$ is 
	a critical point of $\varphi$, then it must hold that 
	\begin{equation}
	\varphi(y_0)=\Bigl(1+\frac{a^2}{b^2}\Bigr)\cos^2(2ay_0)\,.
	\end{equation}
	Let $t_0:=\cos(2ay_0)$. If $t_0\leq0$, then \eqref{cr.pt.eqn} 
	implies that $\sin(2ay_0)<0$. Then $\varphi(y_0)<1+e^{-4by_0}\leq2$ 
	and~\eqref{suff.cond.phi} holds.
	If $b\geq a t_0$, then 
	$\varphi(y_0)\leq 1+t_0^2\leq2$ and \eqref{suff.cond.phi} 
	follows. If $\sin(2ay_0)<0$, then again $\varphi(y_0)<2$, 
	implying~\eqref{suff.cond.phi}. So let us assume that 
	$t_0>0$, $at_0>b$ and $\sin(2ay_0)\geq0$. 
	Then $\sin(2ay_0)=\sqrt{1-t_0^2}$ and \eqref{cr.pt.eqn} 
	implies that 
	\begin{equation}
	0<\frac{at_0}{b}-1\leq\sqrt{1-t_0^2}\,.
	\end{equation}
	Squaring both sides of the latter, we conclude 
	\begin{equation}\label{constr.10}
	0<t_0\leq\frac{2ab}{a^2+b^2}\,.
	\end{equation}
	On the other hand, the condition $at_0>b$ together with 
	\eqref{constr.10} imply that $b\leq a$. Therefore, in 
	this case, we have
	\begin{equation}\label{cute.obsv}
	\varphi(y_0)=t_0^2\Bigl(1+\frac{a^2}{b^2}\Bigr)
	\leq \frac{4a^2}{a^2+b^2}= 2\bigg(1+\frac{|a^2-b^2|}{a^2+b^2}\bigg)\,.
	\end{equation}
	Finally, by noticing that $\varphi(y)\to2$ as $y\downarrow 0$ and that 
	$\varphi(y)\to1$ as $y\uparrow +\infty$, and summing up the above observations, 
	we conclude \eqref{suff.cond.phi}.

	\medskip
	\noindent
	\textbf{Case} $x\geq 0$, $y<0$.
	In this case, we have
	\begin{equation}
	|G_\la(x,y)|=\frac{e^{ay-bx}}{2|\alpha|\sqrt{|\la|}}
	\leq \frac{1}{\sqrt{2|\la|}}\,.
	\end{equation}

	\medskip
	\noindent
	\textbf{Case} $x<0$, $y\geq0$.
	In this case, the result follows from the previous 
	step and the observation \eqref{assym.Green.fn}.
	
	\medskip
	\noindent
	\textbf{Case} $x<0$, $y<0$.
	In view of \eqref{assym.Green.fn}, there is no
	loss of generality in assuming that $x\leq y$.
	Then elementary calculations show that 
	\begin{equation}
	|G_\la(x,y)|^2 = \frac{1}{4|\la|} \Psi(x,y)\,,
	\end{equation}
	where 
	\begin{equation}
	\Psi(x,y):= e^{-2a(y-x)}+e^{2a(y+x)}-2e^{2ax}\sin(2by)\,. 
	\end{equation}
	In the same way as in the first case, by fixing $y\geq0$ 
	and considering~$\Psi$ as a function of the variable $x$ on the 
	interval $(-\infty, y]$, we come to the conclusion that 
	\begin{equation}\label{suff.cond.psi}
	\max_{x\leq y\leq0}\Psi(x,y) 
	\leq 2\max\bigg(1+\frac{|\Re(\la)|}{|\la|}\bigg)\,,
	\end{equation}
	completing the proof.
\end{proof}	

\subsection{Krein resolvent formula}

Let $B_{\pm}=\mp \partial_x^2$ on $ L^2(\R_{\pm})$ with 
Dirichlet boundary conditions at the origin. We set  
\begin{align*}
B_0:=\begin{pmatrix}
B_+&0\\0&B_-
\end{pmatrix}
\quad \mbox{on}\quad L^2(\R)=L^2(\R_+)\oplus L^2(\R_-)
\end{align*}
and 
\begin{align*}
F_{\lambda}:=\frac{1}{\I(\sqrt{\la}+\sqrt{-\la})}
|f_{\la}\rangle\langle J f_{\la}|\,,
\end{align*}
where $J=\sgn(\cdot)$ and 
\begin{align*}
f_{\la}(x):=\e^{\I\sqrt{\la}}\mathbf{1}_{x>0}+\e^{-\I\sqrt{-\la}}\mathbf{1}_{x<0}\,.
\end{align*}
We use the following Krein type resolvent formula (see \cite{Beh-Phil-Tru-2013}):
\begin{align}\label{Krein-type resolvent formula}
(H_0-\la)^{-1}=(B_0-\la)^{-1}-F_{\la}\,.
\end{align}
The operator $F_{\la}$ has rank one, and its unique eigenvalue is given by
\begin{align*}
\mu(\la)=\frac{1}{2\I(\sqrt{\la}+\sqrt{-\la})}
\left(\frac{1}{\Im(\sqrt{\la})}-\frac{1}{\Im(\sqrt{-\la})}\right)\,.
\end{align*}
We then have 
\begin{align*}
\|F_{\la}\|=\|F_{\la}\|_{\rm tr}=|\mu(\la)|\leq \frac{1}{|\Im(\la)|}\,,
\end{align*}
Together with the self-adjointness of $B_0$ and 
\eqref{Krein-type resolvent formula} this implies that
\begin{align}\label{L2 resolvent bound H0}
\|(H_0-\lambda)^{-1}\|\leq \frac{2}{|\Im(\la)|}\,.
\end{align}
\section
{Limiting absorption type principle}
\label{Sec:BS.principle}
\noindent
The main role in the proofs of Theorems~\ref{thm:ell1}-\ref{thm:p-norm.1} 
is played by the Birman-Schwinger operator
%
\begin{equation}\label{BS.op}
K_\la := |V|^{1/2} \, (H_0-\la)^{-1} \, V_{1/2}
\qquad \mbox{with} \qquad
V_{1/2} := |V|^{1/2} \, \sgn(V) 
\,,
\end{equation}
where $\sgn\CC\to\CC$ is the complex signum function 
defined by $\sgn(z):=z/|z|$ for $z\neq0$ with the 
convention $\sgn(0):=0$. The operator~$K_\la$ is well 
defined on its natural domain of the composition of 
three operators for all $\la\in\Com$.
Furthermore, we have a useful formula for its integral kernel 
\begin{equation}\label{K.op}
K_\la(x,y) = |V|^{1/2}(x) \, G_\la(x,y) \, V_{1/2}(y)\,,
\end{equation}
where $G_\la$ is the Green's function of~$H_0-\la$.

The following lemma can be considered as a one-sided version 
of the conventional Birman-Schwinger principle extended 
to possibly embedded eigenvalues. 
Its proof is {heavily} inspired by the ones of the 
analogous results 
in~\cite{Fan-Kre-Veg-JST18, Fra-Sim-JST-17, Ibr-Kre-Lap}.

\begin{Lemma}\label{Lem.BS}
	Assume that $V\in L^1(\Real)$. Let $\la\in\sigma_\mathrm{p}(H_V)\cap\bigl(\CC^+\cup\RR\bigr)$ 
	and $\psi\in\Dom(H_V)$ 
	be an associated eigenvector. Then $\phi:=|V|^{1/2}\psi \in L^2(\RR)$ and 
	\begin{equation}\label{BS}
	\forall \varphi \in L^2(\RR)\,, \quad
	\lim_{\eps \to 0^+}	(\varphi, K_{\la+i\eps} \phi) 
	= - (\varphi,\phi)\,.
	\end{equation}
	In particular, 
	\begin{equation}\label{BS.cor}
	\liminf_{\eps\to 0^+} \|K_{\la+i\eps}\|\geq1\,.
	\end{equation}

\end{Lemma}
\begin{proof}
	It readily follows from \eqref{rel.bddness}	that $\phi\in L^2(\RR)$.
	We fix $\varphi \in L^2(\RR)$. Given any $\la\in\sigma_\mathrm{p}(H_V)$, 
	we have $z:=\la + i\eps \in\CC^+\subset\rho(H_0)$ for all $\eps>0$. 
	Furthermore,  
	\begin{equation}\label{i1}
	\begin{aligned}
	(\varphi,K_{z} \phi)
	&= \iint_{\Real\times\Real} 
	\ov{\varphi(x)} \, |V|^{1/2}(x) \, 
	G_{z}(x,y) \, V(y) \, \psi(y) 
	\, \rd x \, \rd y\\
	&= \int_{\Real} \eta_\eps(y) \, V(y) \, \psi(y) \sgn(y)\,\rd y
	\,,
	\end{aligned}
	\end{equation}
	where
	\begin{equation*}
	\eta_\eps :=\sgn(\mydot)\bigg(\int_{\Real} 
	\ov{\varphi(x)} \, |V|^{1/2}(x) \, 
	G_{z}(x,\mydot) \, \rd x \bigg)
	= (H_0-z)^{-1} \, |V|^{1/2} \, \ov{\varphi}\sgn
	\,.
	\end{equation*}
	Here the second equality holds due to the (anti-)symmetric 
	property of the Green's function in \eqref{Green.fn}: 
	$G_z(x,y)=G_z(y,x)$ if $\sgn(x)=\sgn(y)$ and 
	$G_z(x,y)=-G_z(y,x)$ if $\sgn(x)=-\sgn(y)$. 
	By the Cauchy-Schwarz inequality, $|V|^{1/2}\ov{\varphi}\in L^2(\RR)$.
	Since $z\notin\sigma(H_0)$, we have $\eta_\eps\in\Dom(H_0)\in H^1(\RR)$, 
	and the weak formulation of the eigenvalue equation~$H_V\psi=\la\psi$ 
	yields
	\begin{equation}\label{i2}
	\begin{aligned}
	\int_{\RR}\eta_\eps(y)V(y)\psi(y)\sgn(y)\,\rd y
	&= -(\ov{\eta_\eps}',\psi')+\la\,(\ov{\eta_\eps},\psi\sgn)
	\\
	&= -(\ov{\psi}',\eta_\eps')+\la\,(\ov{\psi}\sgn,\eta_\eps)
	\\
	&= -(\ov{\psi}',\eta_\eps')+z\,(\ov{\psi}\sgn,\eta_\eps)
	-i\eps\,(\ov{\psi}\sgn,\eta_\eps)
	\\
	&= -(\ov{\psi},|V|^{1/2}\ov{\varphi}) 
	-i\eps\,(\ov{\psi}\sgn,\eta_\eps)
	\\
	&= -(\varphi,|V|^{1/2}\psi) 
	- i\eps\,(\ov{\eta_\eps},\psi\sgn)
	\,.
	\end{aligned}
	\end{equation}
	Here the penultimate equality follows from 
	the weak formulation of the resolvent equation
	$(H_0-z)\eta_\eps = |V|^{1/2}\ov{\varphi}\sgn$.
	Consequently, \eqref{i1} and~\eqref{i2} 
	imply~\eqref{BS} after taking the limit 
	$\eps \to 0^+$, provided that 
	$\eps \, (\bar\eta_\eps,\psi\sgn) \to 0$
	as $\eps \to 0^+$. To see the latter, we write
	\begin{equation*}
	|(\ov{\eta_\eps},\psi\sgn)| 
	= |(\varphi\sgn,M_\eps\psi\sgn)|
	\leq \|\varphi\| \;\! \|M_\eps\| \;\! \|\psi\|
	\,,
	\end{equation*}
	where
	$M_\eps:=|V|^{1/2} (H_0-z)^{-1}$,
	and it remains to be shown that $\eps \, \|M_\eps\|$ 
	tends to zero as $\eps \to 0^+$.
	To this end, first we notice that 
	\begin{equation}
	\begin{aligned}
	\sup_{x\geq0}\int_{\RR}|G_z(x,y)|^2\, dy 
	&\leq
	\sup_{x\geq0}\int_{-\infty}^0|G_z(x,y)|^2\, dy 
	+
	\sup_{x\geq0}\int_0^{\infty}|G_z(x,y)|^2\,\rd y\\
	&\leq \frac{1}{2|z|}\int_{-\infty}^{0}e^{2y\Re(\sqrt{z})}\,\rd y
	+ \frac{1}{|z|}\int_0^{\infty}e^{-2y\Im(\sqrt{z})}\,\rd y\\
	&=\frac{1}{4|z|}\bigg[\frac{1}{\Re(\sqrt{z})}+\frac{2}{\Im(\sqrt{z})}\bigg]
	\end{aligned}
	\end{equation}
	and that
	\begin{equation}
	\begin{aligned}
	\sup_{x\leq0}\int_{\RR}|G_z(x,y)|^2\,\rd y 
	&\leq
	\sup_{x\leq0}\int_{-\infty}^0|G_z(x,y)|^2\,\rd y 
	+
	\sup_{x\leq0}\int_0^{\infty}|G_z(x,y)|^2\,\rd y\\
	&\leq \frac{1}{|z|}\int_{-\infty}^{0}e^{2y\Re(\sqrt{z})}\,\rd y
	+ \frac{1}{2|z|}\int_0^{\infty}e^{-2y\Im(\sqrt{z})}\,\rd y\\
	&=\frac{1}{4|z|}\bigg[\frac{2}{\Re(\sqrt{z})}+\frac{1}{\Im(\sqrt{z})}\bigg]\,.
	\end{aligned}
	\end{equation}
	Therefore, it follows that 
	\begin{equation}
	\begin{aligned}
	\|M_\eps\|^2 \leq \|M_\eps\|_\mathrm{HS}^2
	&= \iint_{\RR \times \RR} 
	|V(x)|\,|G_z(x,y)|^2\,\rd x\,\rd y\\
	&\leq \sup_{x\in\RR}\,\int_{\RR}|G_z(x,y)|^2\,\rd y \,
	\int_\RR |V(x)|\,\rd x \\
	&\leq \frac{1}{2|z|}\bigg[\frac{1}{\Re(\sqrt{z})}+\frac{1}{\Im(\sqrt{z})}\bigg]
	\int_\RR |V(x)| \,\rd x\,.
	\end{aligned}
	\end{equation}
	On the other hand, elementary calculations show that
	\begin{equation}
	|z|\Re(\sqrt{z}) \sim
	\begin{cases}
	\eps^{1/2} & \mbox{if} \quad \la = 0 \,,
	\\
	1 & \mbox{otherwise} \,,
	\end{cases}
	\end{equation}
	while
	\begin{equation}
	|z|\Im(\sqrt{z}) \sim
	\begin{cases}
	\eps^{1/2} & \mbox{if} \quad \la = 0 \,,
	\\
	\eps & \mbox{if} \quad \Re(\la) > 0 \ \& \ \Im(\la) = 0 \,,
	\\
	1 & \mbox{otherwise} \,.
	\end{cases}
	\end{equation}
	Hence, we have $\|M_\eps\|=\mathcal{O}(|\eps|^{-3/4})$ 
	as $\eps\to0^+$, which concludes the proof of the claim 
	in~\eqref{BS}.
	
	\smallskip
	
	Now applying \eqref{BS} with $\varphi=\phi$ and taking the 
	limit~$\eps\to0^+$, we obtain obtain 
	\begin{equation}
	|(\phi,\phi)| \leq |(\phi,K_{\la+i\eps}\phi)| 
	\leq 
	\liminf_{\eps\to 0^+} \|K_{\la+i\eps}\| \|\phi\|^2
	\,.
	\end{equation}
	Since $\phi=|V|^{1/2}\psi\neq0$ (if this were 
	not true, then $\la$ would be an eigenvalue for $H_0$, 
	unless $\psi=0$, which is impossible), we thus 
	conclude~\eqref{BS.cor}.
\end{proof}
%
%
\section{Proofs of the main results}
\label{Sec:proofs} 
Without loss of generality we work with fixed 
$\la\in\Com^+\cup\RR$ in the sequel. The key strategy 
of the proof is to estimate the norm 
the Birman-Schwinger operator from above and apply 
Lemma~\ref{Lem.BS}.  
\subsection{Proof of Theorem~\ref{thm:ell1}}

First, we consider the case $\la\in\CC^+$. 
In view of Lemma~\ref{Lem.Resv.estim}, we can 
estimate the norm of the Birman-Schwinger operator 
as follows
\begin{equation}\label{Est.BS.op}
\begin{aligned}
\|K_\lambda\|^2 \leq \|K_\lambda\|^2_\mathrm{HS}
&=\iint_{\RR\times\RR}|V(x)|\,
|G_\lambda(x,y)|^2
\,|V(y)|\,\d x\,\d y\\
& \leq \frac{1}{2}\bigg(\frac{1}{|\la|}+\frac{|\Re(\la)|}{|\la|^2}\bigg)\|V\|_1^2\,.
\end{aligned}
\end{equation}
If $\la\in\RR\setminus\{0\}$, then the same analysis 
applied for $\lambda+i\eps$ with $\eps>0$ (instead 
of $\lambda$) yields
\begin{equation*}
\begin{aligned}
\liminf_{\eps\to 0^+} \|K_{\lambda+i\eps}\|^2
&\leq \liminf_{\eps\to 0^+} \frac{1}{2}\bigg(\frac{1}{|\la+i\eps|}+\frac{|\Re(\la)|}{|\la+i\eps|^2}\bigg)\|V\|_1^2\\
&=\frac{1}{2}\bigg(\frac{1}{|\la|}+\frac{|\Re(\la)|}{|\la|^2}\bigg)\|V\|_1^2\,.
\end{aligned}
\end{equation*}
Hence, Lemma~\ref{Lem.BS} implies that
if $\la\in\bigl(\CC^{+}\cup\RR\big)\setminus\{0\}$
is an eigenvalue for $H_V$, then we must have
\[
2\leq\bigg(\frac{1}{|\la|}+\frac{|\Re(\la)|}{|\la|^2}\bigg)\|V\|_1^2\,,
\] 
\ie~\eqref{encl.ell1} must hold. 
This completes the proof since \eqref{encl.ell1} trivially holds 
for $\la=0$.
\qed
%
\subsection{Optimality of the eigenvalue bound of Theorem~\ref{thm:ell1}}
\label{Sect.optim}
\noindent
Here we demonstrate that the result~\eqref{encl.ell1} 
is sharp in the sense that to any non-real boundary point of the 
spectral enclosure, there exists a delta--potential~$V$ so that 
this boundary point is an eigenvalue of~$H_{V}$. By standard approximation arguments, it follows that there exists a sequence of $L^1$-potentials $V_n$ such that the eigenvalues of $H_{V_n}$ converge to those of $H_V$. This shows that that the boundary curve in \eqref{encl.ell1} cannot be improved.

Let us take an arbitrary non-real boundary
point~$\la$ of the spectral enclosure. Since the boundary curve
is symmetric with respect to the real axis, there is no loss of
generality in assuming that $\la\in\Com^+$. Let us denote the 
positive numbers $\Re(\sqrt{\la})$ and $\Im(\sqrt{\la})$ by 
$a$ and $b$, respectively. 
Further, for given $Q>0$, let us consider $\alpha\in\Com$
with $|\alpha|=Q$ and the operator $H_{V}$ with the Dirac delta potential 
$V(x)=\alpha\delta(x-x_0)$, $x\in\Real$,
where
\begin{equation}
x_0=\frac{1}{2a}\arccos\Bigl(\frac{2ab}{a^2+b^2}\Bigr)\,.
\end{equation}
The operator $H_V$ can be defined rigorously by form methods (see Section \ref{Sec:op.def}). In this case, the Birman-Schwinger operator reduces to the 
multiplication operator with the constant function 
$\alpha G_\la(x_0,x_0)$, where $G_\la$ is the Green's function 
defined in \eqref{Green.fn}, and the inequality \eqref{Est.BS.op}
becomes equality. 
Furthermore, we have $Q|G_\la(x_0,x_0)|=1$ (see the proo of 
Lemma~\ref{Lem.Resv.estim}). 
Hence, by fixing the phase of $\alpha$ in 
such a way that $\alpha G_\la(x_0,x_0)=1$, we deduce from 
the Birman-Schwinger principle that $\la\in\sigma(H_V)$.
On the other hand, we have $\sigma_{\text{ess}}(H_V)=\RR$ 
(since the perturbation is a point interaction)
and $\sigma_{\text{r}}(H_V)=\emptyset$ (since $H_V$ is 
$J$-self-adjoint, where $J$ is the complex conjugation operator).
Therefore, $\la\in\Com^+$ must be a discrete eigenvalue for $H_V$. 
\qed
\subsection{Proof of Theorem~\ref{thm:p-norm}}
The proof uses complex interpolation and the 
result of Theorem~\ref{thm:ell1}. Observe that 
\eqref{encl.ellp} holds trivially for $\la\in\Real$.
For $\la\in\Com^+$, let us consider the operator family 
\[
T_{z}:=|V|^{z p/2}(H_0-\la)^{-1}|V|^{z p/2},
\]
for $z\in\CC$ with $0\leq\Re z\leq 1$. 
First, we note that $T_1$ is a bounded operator 
under our hypothesis on $V$. Indeed, $|V|^{p/2}$ maps
$L^2(\Real)$ to $H^{-1}(\Real)$ by duality and 
$(H_0-\la)^{-1}$ is an isomorphism between $H^{-1}(\Real)$ 
and $H^1(\Real)$, while the latter space is mapped by 
$|V|^{p/2}$ back to $L^2(\Real)$.
Further, we note that $T_{z}$ is 
continuous in the closed strip $0\leq\Re z\leq 1$, analytic 
in its interior and we have  
\[
\sup_{0\leq\Re z\leq 1}\|T_{z}\|\leq\max\Bigl\{\frac{2}{|\Im(\la)|}, \|T_1\|\Bigr\}\,,
\]
see \eqref{L2 resolvent bound H0}. In particular, $T_{z}$ is uniformly bounded for $0\leq\Re z\leq 1$. 
Since $V\in L^p(\RR)$ by the hypothesis, we can 
proceed as in the proof of Theorem~\ref{thm:ell1} and conclude 
\[
\|T_{1+\ii y}\|\leq\||V|^{p/2}(H_0-\la)^{-1}|V|^{p/2}\| 
\leq 
\frac{\sqrt{|\la|+|\Re(\la)|}}{\sqrt{2}|\la|}
\|V\|_p^p\,,
\]
for any $y\in\RR$. Moreover, for all $y\in\RR$, we have also 
the trivial estimate
\[
\|T_{\ii y}\|\leq 
\frac{2}{|\Im(\la)|}\,,
\]
see \eqref{L2 resolvent bound H0}. Thus, Stein's complex interpolation theorem 
(see \eg~\cite[Thm.~V.4.1]{Ste-Wei-1971}) yields the following bound 
for the Birman-Schwinger operator
\begin{equation}\label{estim.interp}
\|K(\la)\|\leq \|T_{1/p}\|
\leq
\frac{(|\la|+|\Re(\la)|)^{\frac{1}{2p}}}{2^{\frac{3}{2p}-1}|\la|^{\frac{1}{p}}}
\frac{\|V\|_p}{|\Im(\la)|^{1-\frac{1}{p}}}\,.
\end{equation}
If $\la$ is an eigenvalue for $H_V$, then the standard Birman-Schwinger 
principle implies that the expression on the right-hand-side of \eqref{estim.interp} cannot be 
strictly less than 1, thus yielding the estimate~\eqref{encl.ellp}. 
\qed
\subsection{Proofs of Corollaries \ref{Cor:ineq.imag.L1} and \ref{Cor:ineq.imag}}
Let $1\leq p <\infty$.
For $\la\in\CC^+$, let $x=\Re(\la)$ and $y=\Im(\la)$.
Then \eqref{encl.ellp} (resp.~\eqref{encl.ell1}) can be written, equivalently, as
\begin{equation}\label{constr.20}
2^{3-2p}(x^2+y^2)y^{2p-2} \leq \bigl(\sqrt{x^2+y^2}+|x|\bigr)\|V\|_p^{2p}\,.
\end{equation}
Since the region corresponding to \eqref{constr.20} is 
symmetric with respect to the imaginary axis, there is no 
loss of generality in assuming that $x\geq0$. Further 
with the change of the variables $x=\|V\|_p^{\frac{2p}{2p-1}}t$, 
$y=\|V\|_p^{\frac{2p}{2p-1}}s$, \eqref{constr.20} reads as
\begin{equation}\label{constr.21}
2^{3-2p}(t^2+s^2)s^{2p-2} \leq \sqrt{t^2+s^2}+t\,.
\end{equation}
Next, let us consider the function 
\begin{equation}
f(x)=\frac{x}{1+x^2}+\frac{x}{\sqrt{1+x^2}}
\end{equation}
for $x\geq0$. Letting $x=\tan(\alpha)$ for $\alpha\in[0,\frac{\pi}{2})$, 
we easily get 
$f(\tan(\alpha))=\sin(\alpha)+\frac{1}{2}\sin(2\alpha)$
which attains its global maximum at $\alpha=\frac{\pi}{3}$.
Consequently, we have 
\[
\sup_{x\geq0}f(x)=\sup_{\alpha\in[0,\frac{\pi}{2}]}f(\tan(\alpha))
=\frac{3\sqrt{3}}{4}=f(\sqrt{3})\,.
\]
Hence, it follows from \eqref{constr.21} that
$2^{3-2p}s^{2p-1}\leq f(s/t)\leq3\sqrt{3}/4$
for all $s\geq0$ and $t>0$. Therefore, 
\begin{equation}
s\leq 2\Bigl(\frac{3\sqrt{3}}{16}\Bigr)^{\frac{1}{2p-1}}\,
\end{equation}
and \eqref{constr.21} becomes equality if and only if 
\[(s,t)=\bigg(2\Bigl(\frac{3\sqrt{3}}{16}\Bigr)^{\frac{1}{2p-1}}, 
\,\frac{2}{\sqrt{3}}\Bigl(\frac{3\sqrt{3}}{16}\Bigr)^{\frac{1}{2p-1}}\bigg)\,,
\]
proving the claim.
\qed
\subsection{Proof of Theorem~\ref{thm:p-norm.1}}
Let $\la\in\Com^+$ and denote by $\Omega$ the support of $V$. 
For arbitrary weight function $\rho>0$, the Schur test yields
\begin{equation*}
\|K(\la)\|\leq\Bigg(\sup_{x\in\Omega}\int_{\Omega}|K_\la(x,y)|
\frac{\d y}{\rho(x,y)}\Bigg)^{1/2}\Bigg(\sup_{y\in\Omega}
\int_{\Omega}|K_\la(x,y)|\rho(x,y)\, \d x\Bigg)^{1/2}\,.
\end{equation*}	
By choosing the weight function 
$$
\rho(x,y):=|V(x)|^{1/2}|V(y)|^{-1/2}, \quad x,y\in\Omega,
$$ 
and using \eqref{assym.Green.fn}, we obtain 
\begin{equation}\label{estim.BS.Schur}
\|K(\la)\|\leq
\sup_{x\in\Omega}\int_{\Omega}|G_\la(x,y)||V(y)|\,\d y
\leq \sup_{x\in\Real}\int_{\Real}|G_\la(x,y)||V(y)|\,\d y\,.
\end{equation}
On the other hand, in view of \eqref{Green.fn}, we have 
\begin{equation*}
\begin{aligned}
\sup_{x\geq0}\int_{\Real}|G_\la(x,y)||V(y)| \,\d y 
&\leq\sup_{x\geq0}\int_{-\infty}^0|G_\la(x,y)||V(y)|\,\d y 
+ \sup_{x\geq0}\int_0^{\infty}|G_\la(x,y)||V(y)|\,\d y\\
&\leq\frac{1}{\sqrt{2|\la|}}\int_{-\infty}^0 e^{\Re(\sqrt{\la})y}|V(y)|\,\d y
+\frac{1}{\sqrt{|\la|}}\int_0^{\infty} e^{-\Im(\sqrt{\la})y}|V(y)|\,\d y\,.
\end{aligned}
\end{equation*} 
By the H\"older inequality, 
\begin{equation*}
\begin{aligned}
\int_{-\infty}^0 e^{\Re(\sqrt{\la})y}|V(y)|\,\d y 
\leq 
\|V\|_{p,-}\Bigg(\int_{-\infty}^0 e^{q\Re(\sqrt{\la})y}\,\d y\Bigg)^{1/q}
=\frac{\|V\|_{p,-}}{\sqrt[q]{q\Re(\sqrt{\la})}}
\end{aligned}
\end{equation*}
and, similarly,
\begin{equation*}
\begin{aligned}
\int_0^{\infty} e^{-\Im(\sqrt{\la})y}|V(y)|\,\d y 
\leq\frac{\|V\|_{p,+}}{\sqrt[q]{q\Im(\sqrt{\la})}}\,.
\end{aligned}
\end{equation*}
Therefore, we have
\begin{equation}\label{estim.Schur.11}
\sup_{x\geq0}\int_{\RR}|G_\la(x,y)||V(y)|\,\d y 
\leq 
\frac{1}{\sqrt[q]{q}\sqrt{|\la|}}
\Bigg(\frac{\|V\|_{p,-}}{\sqrt{2}\sqrt[q]{\Re(\sqrt{\la})}}
+\frac{\|V\|_{p,+}}{\sqrt[q]{\Im(\sqrt{\la})}}\Bigg)
\end{equation}
and, analogously,
\begin{equation}\label{estim.Schur.12}
\sup_{x\leq0}\int_{\RR}|G_\la(x,y)||V(y)|\,\d y 
\leq \frac{1}{\sqrt[q]{q}\sqrt{|\la|}}
\Bigg(\frac{\|V\|_{p,-}}{\sqrt[q]{\Re(\sqrt{\la})}}
+\frac{\|V\|_{p,+}}{\sqrt{2}\sqrt[q]{\Im(\sqrt{\la})}}\Bigg).
\end{equation}
Recalling \eqref{estim.BS.Schur}, we conclude that the 
maximum of the quantities on the right-hand-sides of 
\eqref{estim.Schur.11} and \eqref{estim.Schur.12} dominates
the operator norm of the Birman-Schwinger operator $K(\la)$
and the result immediately follows from the 
Birman-Schwinger principle as in the proof 
of Theorem~\ref{thm:p-norm}.
\qed

\subsection{Proof of Theorem~\ref{thm:Wigner-vonNeumann}}
First we establish an auxiliary lemma. The proof is a straightforward calculation and is omitted.
\begin{Lemma}\label{lemma free solution}
	Let $\la>0$ be arbitrary. Then the function
	\begin{align}\label{def.u}
	u(x)=\e^{\sqrt{\la}x}\mathbf{1}_{x\leq 0}+
	\sqrt{2}\sin(\sqrt{\la}x+\pi/4)\mathbf{1}_{x\geq 0}
	\end{align}
	satisfies $H_0u=\la u$ in the weak sense. Moreover, $u\in L^2(\R_-)$ and 
	$u\notin L^2(\R_+)$.
\end{Lemma}
%
%
%
%

\smallskip
\noindent
{\emph{Proof of Theorem~\ref{thm:Wigner-vonNeumann}}}. We first 
discuss the case $n=1$. Set $\psi(x)=u(x)\chi(x)$, with $u$ given by 
\eqref{def.u} and with $\chi\in C^{\infty}(\R)$ to be chosen later. Then
\begin{align*}
(H_0-\lambda) \psi(x)=-\sgn(x) (2u'(x)\chi'(x)+u(x)\chi''(x))\,.
\end{align*}
Selecting 
\begin{align}\label{def. V}
V=2\frac{u'\chi'}{u\chi}+\frac{\chi''}{\chi}\,,
\end{align}
the equation $H_V\psi=\lambda\psi$ is satisfied by definition, 
provided that $V$ is well-defined. The issue is of course 
that~$u$ has zeros on $\R_+$. To cancel these we first set
\begin{align*}
g(x)=\int_0^{x}\sin^2(\sqrt{\la}t+\pi/4)\,\rd t,\quad x\in\RR\,,
\end{align*}
and choose
\begin{align}\label{def. chi}
\chi(x)=(1+g(x)^2)^{-1},\quad x\in\R\,.
\end{align}
We then get
\begin{align*}
V=\frac{8g^2g'^2}{(1+g^2)^2}-\frac{2(g'^2+gg'')}{1+g^2}-\frac{4gg'u'}{u(1+g^2)}\,.
\end{align*}
Since $g'/u$ vanishes on the zero set of $u$ and $g(x)\geq c|x|$ for 
some $c>0$ and all sufficiently large $x$, we see that~$V$ 
satisfies~\eqref{bound Vn} for $n=1$. For arbitrary $n\in\N$ we 
replace $\chi$ in \eqref{def. chi} by $(n^2+g(x)^2)^{-1}$.   
\qed
%
%
\subsection{Proof of Theorem~\ref{thm:square.well}}
By scaling \eqref{scaling} with $\rho=\sqrt{\mu}$ we may 
assume that $\mu=1$. We make the following Ansatz for the 
wavefunction $\psi$:
\begin{align*}
\psi(x)=\begin{cases}
\e^{\sqrt{\lambda} x}\quad &x\leq 0,\\
A\e^{\I k x}+B\e^{-\I k x}\quad &0\leq x\leq R,\\
C\e^{\I \sqrt{\lambda} x}\quad &x\geq R,
\end{cases}
\end{align*}
where 
\begin{align}\label{square well 0}
k^2+V_0=\la,\quad \Im(k)>0,\quad \Im(\sqrt{\la})>0,\quad 
\Re(\sqrt{\la})>0
\end{align}
and $R=R(\eps), k=k(\eps)$ will be chosen later. Taking
\begin{align*}
A=A(k,\la)=\frac{1}{2}+\frac{\sqrt{\la}}{2\I k},\quad 
B=B(k,\la)=\frac{1}{2}-\frac{\sqrt{\la}}{2\I k},
\end{align*}
it follows that $\psi,\psi'$ are continuous at $x=0$. 
It is easy to see that there exists $C\in\C$ such that 
$\psi,\psi'$ are continuous at $x=R$ if and only if
\begin{align}\label{square well 1}
\sqrt{\la}=-k \frac{B-A\e^{2\I kR}}{B+A\e^{2\I kR}}\,.
\end{align} 
We set 
\begin{align}\label{square well 1.5}
k=k(\eps)=-1+\I\eps,\quad R=R(\eps)=\frac{|\ln\eps|}{2\eps}+\theta
\end{align}
with $\theta=\theta(\eps)\in [0,\pi]$ to be chosen later. 
Changing variables from $\la$ to $\omega$ v.i.z. 
$\sqrt{\la}=1+\I\eps+\omega$ and setting
\begin{align*}
f_{\eps}(\omega)=1+\I\eps+\omega+k\frac{B-A\e^{2\I kR}}{B+A\e^{2\I kR}},
\quad \omega \in B(0,C\eps^2),
\end{align*}
with $C$ independent of $\eps$ to be chosen sufficiently large,
we see that \eqref{square well 1} is equivalent to $f_{\eps}(\omega)=0$. 
Since $\e^{-2\im k R}=\eps(1+\mathcal{O}(\eps))$ and 
$A/B=(-1+\I)/(1+\I)+\mathcal{O}(\eps)$ as $\eps\to0^+$, 
we obtain, by choosing $\theta(\eps)$ such that 
$(-1+\I)/(1+\I)\e^{-2\I R(\eps)}=-1$,
\begin{align*}
|f_{\eps}(\omega)-\omega|&=
|1+\I\eps+(-1+\I\eps)(1-2A/B\e^{2\I k R}+\mathcal{O}(\eps^2))|\\
&=|2\I\eps+2A/B\eps\e^{-2\I R}+\mathcal{O}(\eps^2)|=\mathcal{O}(\eps^2)
\end{align*}
as $\eps\to0^+$.
It follows that, for $\eps$ sufficiently small and $C$ sufficiently large,
\begin{align*}
|f_{\eps}(\omega)-\omega|<|\omega|,\quad \omega\in\partial B(0,C\eps^2)\,.
\end{align*}
By Rouch\'e's theorem $f_{\eps}$ has exactly one zero in $B(0,C\eps^2)$.
This implies the existence of an eigenvalue $\lambda$ with the claimed properties.
Recalling the relation between $k$ and $V_0$ in \eqref{square well 0} and using 
\eqref{square well 1.5} we get the estimate
\begin{align}\label{norm of Veps}
\|V(\eps)\|_p\approx\eps R(\eps)^{1/p}=\mathcal{O}(\eps^{1-1/p}|\ln\eps|^{1/p})
\end{align}
as $\eps\to0^+$. This proves \eqref{square well limit V}.
\qed

\smallskip
\noindent
\textbf{Sharpness of exponents in \eqref{encl.ellp}:} 
Since $|\la|\approx|\Re(\la)|\approx 1$ in the above example, 
in equality \eqref{encl.ellp} says that
\begin{align*}
\eps^{1-1/p}\leq C\|V(\eps)\|_p\,.
\end{align*}
Comparing this to \eqref{norm of Veps} we see that this is 
essentially sharp (up to logarithms). In particular, 
the exponent $1-1/p$ of $\eps\approx \Im(\la)$ cannot  be made smaller.
\qed

\subsection{Proof of Theorem~\ref{thm:weak.coupling}}
As usual, we split the Birman-Schwinger operator into a singular 
and a regular part (as $\la\to 0$),
\begin{align*}
K_{\la}=L_{\la}+M_{\la}\,.
\end{align*}
We first ignore $M_{\la}$ and concentrate on 
\begin{align*}
L_{\la}=\frac{1}{2\alpha\sqrt{\la}}\left(|f^+\rangle\langle g^+|-|f^+\rangle\langle g^-|+|f^-\rangle\langle g^+|-|f^-\rangle\langle g^-|\right)\,,
\end{align*}
where we set
\begin{align*}
f^{\pm}(x)=\mathbf{1}_{\pm}(x)|V(x)|^{1/2},\quad g^{\pm}(x)=\mathbf{1}_{\pm}(x)V_{1/2}(x),\quad \mathbf{1}_{\pm}=\mathbf{1}_{\R_{\pm}}\,.
\end{align*}
In the basis $\{f_+,f_-\}$, 
\begin{align*}
L_{\la}=\frac{1}{2\alpha\sqrt{\la}}
\begin{pmatrix}
v_+& -v_-\\
v_+&-v_-
\end{pmatrix},
\end{align*}
from which we see that $L_{\la}$ has rank $1$. Then
\begin{align*}
\det(I+\eps L_{\la})=1+\eps\Tr L_{\la}=1+\frac{\eps}{2\alpha\sqrt{\la}}v_{\rm sgn}
\end{align*}
Hence,
\begin{align*}
0\in\sigma(I+\eps L_{\la})\iff\det(I+\eps L_{\la})=0
\iff
-\sqrt{\la}=\frac{\eps}{2\alpha}v_{\rm sgn}\,.
\end{align*}
Since we are assuming that $\Im(\sqrt{\lambda})>0$ the rightmost equality can only hold if $\Re(v_{\rm sgn})+\Im(v_{\rm sgn})<0$. We also see that 
\begin{align*}
\lambda\in \C^+ \iff \Re(\sqrt{\la})>0 \iff \Re(v_{\rm sgn})<\Im(v_{\rm sgn}).\, 
\end{align*}
We now repeat the argument, but this time taking $M_{\lambda}$ into account. The key observation is that the function $(x,y)\mapsto|\lambda||M_{\lambda}(x,y)|^2$ is bounded from above, up to a constant, by the $L^1$-majorant $(x,y)\mapsto |V(x)||V(y)|$. 
%
Clearly, 
\begin{align*}
\lim_{\lambda\to 0}|\lambda||M_{\lambda}(x,y)|^2=0\,.
\end{align*}
Hence, by dominated convergence,
\begin{align}\label{HS norm of Mlambda}
\|M_{\lambda}\|_{\rm HS}=o(|\lambda|^{-1/2}),\quad (\lambda\to 0)\,.
\end{align}
Since we expect that $c^{-1}\eps\leq|\lambda|^{1/2}\leq c\eps$ for $\lambda=\lambda(\eps)$ and some $c>0$, we assume this from now on; we will see later that this assumption is indeed justified. Since then $\eps\|M_{\lambda}\|<1$ for sufficiently small $\eps$, it follows that $(I+\eps M_{\lambda})$ is invertible and 
\begin{align*}
0\in\sigma(I+\eps K_{\lambda})\iff\det(I+\eps (I+\eps M_{\lambda})^{-1} L_{\lambda})=0\,.
\end{align*}
Similarly as before,
\begin{align*}
\det(I+\eps (I+\eps M_{\lambda})^{-1} L_{\lambda})&=1+\eps\Tr (I+\eps M_{\lambda})^{-1}L_{\lambda}\\
&=1+\frac{\eps}{2\alpha\sqrt{\lambda}}v_{\rm sgn}+r(\eps,\lambda)\,,
\end{align*}
where $r(\eps,\lambda)=O(\eps\|M_{\lambda}\|)$.
We now change variables from $\lambda$ to $z=\sqrt{\lambda}$. Clearly, $f_{\eps}(z):=\det(I+\eps (I+\eps M_{\lambda})^{-1} L_{\lambda})$ is an analytic function for $z\in B(-\frac{\eps}{2\alpha}v_{\rm sgn},\rho(\eps))$, provided $\rho(\eps)=o(\eps)$; we may choose $\rho(\eps)=\max_{c^{-1}\eps\leq|\lambda|^{1/2}\leq c\eps}|r(\eps,\lambda)|$. Then we have
\begin{align*}
\Bigl|f_{\eps}(z)-\Bigl(1+\frac{\eps}{2\alpha z}v_{\rm sgn}\Bigr)\Bigr|\leq \rho(\eps),\quad z\in\partial B\Bigl(-\frac{\eps}{2\alpha}v_{\rm sgn},\rho(\eps)\Bigr)\,.
\end{align*}
On the other hand, since $v_{\rm sgn}\neq 0$, we have
\begin{align*}
\Bigl|1+\frac{\eps}{2\alpha z}v_{\rm sgn}\Bigr|\geq \frac{2|\alpha|\rho(\eps)}{\eps|v_{\rm sgn}|+2|\alpha|\rho(\eps)}\geq \frac{\rho(\eps)}{\eps|v_{\rm sgn}|},\quad z\in\partial B\Bigl(-\frac{\eps}{2\alpha}v_{\rm sgn},\rho(\eps)\Bigr)\,.
\end{align*}
Hence, for $\eps<1/|v_{\rm sgn}|$, we have
\begin{align*}
\Bigl|f_{\eps}(z)-\Bigl(1+\frac{\eps}{2\alpha z}v_{\rm sgn}\Bigr)\Bigr|<\Bigl|1+\frac{\eps}{2\alpha z}v_{\rm sgn}\Bigr|,\quad z\in\partial B\Bigl(-\frac{\eps}{2\alpha}v_{\rm sgn},\rho(\eps)\Bigr)\,.
\end{align*}
By Rouch\'e's theorem, $f_{\eps}$ has exactly one zero in $B(-\frac{\eps}{2\alpha}v_{\rm sgn},\rho(\eps))$. Since $\rho(\eps)=o(\eps)$ this proves the theorem.
\qed

\smallskip
\noindent
\textbf{Sharpness of \eqref{encl.ell1}:} 
If we assume that $V$ is real-valued and $v_{\rm sgn}<0$, then the assumptions of Theorem~\ref{thm:weak.coupling} are satisfied. If we assume in addition that $V$ is supported either on~$\R_+$ or on~$\R_-$, then $v_{\rm sgn}=-\|V\|_1$. Since $1/\alpha^2=2\I$ we see that
\begin{align*}
\lambda(\eps)=\frac{\I}{2}\|\eps V\|_1^2+o(\eps^2),\quad \Re(\la(\eps))=o(\eps^2)\,.
\end{align*}
This shows that inequality (1.4) is saturated in the limit $\eps\to 0$.
\qed
\subsection{Proof of Theorem~\ref{thm:LT}}
We start with the Schatten bound
\begin{align}\label{Schatten bound}
\|W_1(H_0-\la)^{-1}W_2\|_{\mathfrak{S}^{2p}}
\leq 
C_p|\Im(\la)|^{1-\frac{1}{p}}|\la|^{-\frac{1}{2p}}\|W_1\|_{2p}\|W_2\|_{2p}\,,
\quad 1\leq p\leq \infty\,,
\end{align}
%
which is a consequence of \eqref{L2 resolvent bound H0} and the pointwise bound for the resolvent kernel. Here and in the following we use the notation
\begin{align*}
\|A\|_{\mathfrak{S}^p}^p:=\sum_{j}s_j(A)^p
\end{align*}
where $s_j(A)$ are the singular numbers of the compact operator $A$.
We make the conformal transformation $z=\lambda^2\in \C\setminus[0,\infty)$ and apply Theorem 3.1 in \cite{Frank-TAMS-18} to the analytic family $z\mapsto K(\lambda(z))$, which by \eqref{Schatten bound} satisfies the bound
\begin{align*}
\|K(\lambda(z))\|_{\mathfrak{S}^{2p}}\leq C_p|\Im(\sqrt{z})|^{1-\frac{1}{p}}|z|^{-\frac{1}{4p}}\|V\|_{p}\,.
\end{align*}
In particular, for $p=1$, we have
\begin{align*}
\|K(\lambda(z))\|_{\mathfrak{S}^2}\leq C|z|^{-\frac{1}{4}}\|V\|_{p}\,.
\end{align*}
By \cite[Theorem~3.1]{Frank-TAMS-18}, for any $\eps>0$, there exists $C$ such that
\begin{align*}
\sum_j \delta(z_j)|z_j|^{-\frac{1}{2}+\frac{1}{2}(-\frac{1}{2}+\eps)_+}
\leq C\|V\|_{1}^{2(1+(-\frac{1}{2}+\eps)_+)}\,.
\end{align*}
Taking $\eps=1/2$ and using the distortion bound $|\Im(\la)|\approx \delta(z_j)|z_j|^{-\frac{1}{2}}$ yields the claim.
\qed
\subsection{Proof of Theorem~\ref{thm:nr.of.EV}}
%
%
We write $\la=k^2$ for the spectral parameter, where $k\in\C\setminus\{0\}$. This amounts to a double covering of the punctured complex $\la$-plane, with $\Im(k)<0$ corresponding to the second (unphysical) sheet. 
From \eqref{G.la.Phi}--\eqref{Phi} we have that
\begin{align*}
|G(\la(k))|^2\leq \frac{1}{|\la|}\e^{2\Im(k)_-(|x|+|y|)}\,,
\end{align*}
which implies
\begin{align*}
\|K(\lambda(k))\|_{\mathfrak{S}^2}\leq \frac{1}{|k|}\int_{\R}\e^{2\Im(k)_-|x|}|V(x)|\,\rd x\,.
\end{align*}
This is analogous to \cite[Proposition~4.1]{Fra-Lap-Saf-2016}, and the proof follows from the same arguments as in \cite{Fra-Lap-Saf-2016}.
\qed

\section*{Acknowledgement}
\noindent
The authors are grateful to Professor Gian-Michele Graf for stimulating discussions.

{\small
	\bibliographystyle{acm}
	\bibliography{bib_indef_SL}
}
\end{document}